\documentclass{elsarticle}
\usepackage{amssymb, amsmath, latexsym}

\renewcommand{\baselinestretch}{\baselinestretch}
\renewcommand{\baselinestretch}{1.1}
\numberwithin{equation}{section}

\newtheorem{thm}{Theorem}[section]
\newtheorem{lem}[thm]{Lemma}
\newtheorem{cor}[thm]{Corollary}
\newtheorem{prop}[thm]{Proposition}

\newdefinition{defn}[thm]{Definition}

\newdefinition{rmk}[thm]{Remark}
\newdefinition{exam}[thm]{Example}
\newproof{proof}{Proof}

\numberwithin{equation}{section}

\newcommand{\ord}{\text{ord}}

\newcommand{\z}{{\mathbb Z}}
\newcommand{\q}{{\mathbb Q}}

\newcommand{\norm}{\mathfrak n}

\newcommand{\ba}{\mathbf a}
\newcommand{\bb}{\mathbf b}
\newcommand{\bc}{\mathbf c}
\newcommand{\bu}{\mathbf u}
\newcommand{\bv}{\mathbf v}
\newcommand{\bx}{\mathbf x}
\newcommand{\by}{\mathbf y}
\newcommand{\bz}{\mathbf z}
\newcommand{\bw}{\mathbf w}
\newcommand{\be}{\mathbf e}
\begin{document}

\begin{frontmatter}

\title{The representation of integers by positive ternary quadratic polynomials}

\author{Wai Kiu Chan \corref{cor1}}

\author{James Ricci \corref{cor2}}

\cortext[cor1]{Author: Wai Kiu Chan, Wesleyan University, Department of Mathematics and Computer Science, 265 Church St., Middletown, CT, 06459, USA; Email, wkchan@wesleyan.edu; Phone, +001 860 685 2196}
\cortext[cor2]{Corresponding Author: James Ricci, Daemen College, Department of Mathematics and Computer Science, 4380 Main St., Amherst, NY, 14226, USA; Email, jricci@daemen.edu; Phone, +001 716 566 7833}

\begin{abstract}
An integral quadratic polynomial is called regular if it represents every integer that is represented by the polynomial itself over the reals and over the $p$-adic integers for every prime $p$.  It is called complete if it is of the form $Q(\bx + \bv)$, where $Q$ is an integral quadratic form in the variables $\bx = (x_1, \ldots, x_n)$ and $\bv$ is a vector in $\q^n$.  Its conductor is defined to be the smallest positive integer $c$ such that $c\bv \in \z^n$.  We prove that for a fixed positive integer $c$, there are only finitely many equivalence classes of positive primitive ternary regular complete quadratic polynomials with conductor $c$.  This generalizes the analogous finiteness results for positive definite regular ternary quadratic forms by Watson \cite{watson1, watson2} and for ternary triangular forms by Chan and Oh \cite{co2}.
\end{abstract}


\end{frontmatter}

\section{Introduction}

Let $f(\bx) = f(x_1, \ldots, x_n)$ be an $n$-ary quadratic polynomial in variables $\bx = (x_1, \ldots, x_n)$ with rational coefficients.  It takes the form
$$f(\bx) = Q(\bx) + \ell(\bx) + m$$
where $Q(\bx)$ is a quadratic form, $\ell(\bx)$ is a linear form, and $m$ is a constant.  Given a rational number $a$, it follows from the Hasse Principle that the diophantine equation
\begin{equation} \label{basicequation}
f(\bx) = a
\end{equation}
is soluble over the rationals if and only if it is soluble over the $p$-adic numbers for each prime $p$ and over the reals.  However, this local-to-global approach breaks down when we consider integral representations.  Indeed, there are plenty of  examples of quadratic polynomials for which \eqref{basicequation} is soluble over each $\z_p$ and over $\mathbb R$, but not soluble over $\z$.  Borrowing a term coined by Dickson for quadratic forms, we call a quadratic polynomial $f(\bx)$ {\em regular} if for every  rational number $a$,
\begin{equation*}
\mbox{\eqref{basicequation} is soluble over $\z$ $\Longleftrightarrow$ \eqref{basicequation} is soluble over each $\z_p$ and over $\mathbb R$.}
\end{equation*}

G.L. Watson \cite{watson1, watson2} showed that there are only finitely many equivalence classes of primitive positive definite regular ternary quadratic forms.  A list of representatives of these classes has been compiled in \cite{jks} by Jagy, Kaplansky, and Schiemann.  Their list contains 913 ternary quadratic forms, and 891 are verified by them to be regular.  Later B.-K. Oh \cite{oh} proves the regularity of 8 of the remaining 22 quadratic forms.  More recently, R. Lemke Oliver \cite{oliver} establishes the regularity of the last 14 quadratic forms under the generalized Riemann Hypothesis.  Watson's result has been generalized by different authors to definite ternary quadratic forms over other rings of arithmetic interest \cite{cd, ci}, to higher dimensional representations of positive definite quadratic forms in more variables \cite{co}, and to positive definite ternary quadratic forms which satisfy other regularity conditions \cite{ce}.

There are regular quadratic polynomials that are not quadratic forms.  A well-known example is the sum of three triangular numbers
$$\frac{x_1(x_1 + 1)}{2} + \frac{x_2(x_2 + 1)}{2} + \frac{x_3(x_3 + 1)}{2},$$
which is universal (i.e. representing all positive integers) and hence regular.   Given positive integers $a, b, c$, we follow the terminology in \cite{co3} and call the polynomial
$$\Delta(a, b, c): = a\frac{x_1(x_1 + 1)}{2} + b\frac{x_2(x_2 + 1)}{2} + c\frac{x_3(x_3 + 1)}{2}$$
a triangular form.  It is {\em primitive} if $\gcd(a, b, c) = 1$.  There are seven universal ternary triangular forms--hence all are regular-- and they were found by Liouville in 1863 \cite{l}.  An example of a regular ternary triangular form which is not universal is $\Delta(1,1,3)$.  We offer a proof in the following example.

\begin{exam}
It is easy to see that $\Delta(1,1,3)$ does not represent 8.  Hence it is not universal.

A positive integer $n$ is represented by $\Delta(1,1,3)$ if and only if $8n + 5$ is represented by $h_1(\bx) = x_1^2 + x_2^2 + 3x_3^2$ with the extra conditions that $x_1 \equiv x_2 \equiv x_3 \equiv 1$ mod 2.  Let $h_2(\bx)$ be the quadratic form $x_1^2 + x_2^2 + 12x_3^2$.  If $r(n)$ is the number of representations of $n$ by $\Delta(1,1,3)$, then
$$r(n) = r_1(8n + 5) - r_2(8n + 5)$$
where, for $i = 1, 2$,  $r_i(8n + 5)$ is the number of representations of $8n + 5$ by $h_i(\bx)$.  Note that both $h_1(\bx)$ and $h_2(\bx)$ have class number 1, and so the Minkowski-Siegel mass formula \cite[Theorem 6.8.1]{k} implies that both $r_1(8n + 5)$ and $r_2(8n + 5)$ can be expressed as products of local densities.   For each odd prime $p$, $h_1(\bx)$ and $h_2(\bx)$ are equivalent over $\z_p$ and hence the local densities $\alpha_p(8n + 5, h_1)$ and $\alpha_p(8n + 5, h_2)$ are the same.   It then follows from \cite[Theorem 6.8.1]{k} that
$$\frac{r_1(8n + 5)}{r_2(8n + 5)} = \frac{\alpha_2(8n + 5, h_1)}{2^{-1}\alpha_2(8n + 5, h_2)}.$$
The 2-adic densities $\alpha_2(8n+5, h_1)$ and $\alpha_2(8n+5, h_2)$ can be computed by \cite[Proposition 5.6.1]{k},  and both of them can be shown to be equal to 16.  Therefore, $r_1(8n + 5) = 2r_2(8n + 5)$ and hence $r(n) = \frac{1}{2}r_1(8n + 5)$.

Now, suppose that $n$ is represented by $\Delta(1,1,3)$ over $\z_p$ for every prime $p$.  Then, since $h_1(\bx)$ has class number 1, $8n + 5$ is represented by $h_1(\bx)$.  This means that $r_1(8n + 5)$ is not zero, whence $r(n)$ is also not zero.  Thus, $n$ is represented by $\Delta(1,1,3)$.  This proves that $\Delta(1,1,3)$ is regular.
\end{exam}

It is shown in \cite{co2} that there are only finitely many primitive ternary regular triangular forms.   In this paper, we will extend Watson's finiteness results to ternary quadratic polynomials with positive definite quadratic parts.

A quadratic polynomial is called {\em nondegenerate} if its quadratic part is nondegenerate.   Let $f(\bx)$ be a nondegenerate quadratic polynomial.  If $Q(\bx)$ is its quadratic part and $B$ is the bilinear form corresponding to $Q$, then there exists a unique $\bv \in \q^n$ such that $2B(\bv, \bx)$ is the linear part of $f(\bx)$.  The {\em conductor} of $f(\bx)$, introduced in \cite{ah}, is the smallest positive integer $c$ such that $c\bv \in \z^n$.  The polynomial is {\em integral} if $f(\ba) \in \z$ for all $\ba \in \z^n$, and is called {\em primitive} if the ideal generated by the set $\{f(\ba) : \ba \in \z^n\}$ is $\z$.  If $Q$ is positive definite, then $f(\mathbf x)$ attains an absolute minimum on $\mathbb Z^n$.  We then call $f(\mathbf x)$ {\em positive} if its absolute minimum  is nonnegative.

Another quadratic polynomial $g(\bx)$ is said to be {\em equivalent} to $f(\bx)$ if there exist $T\in \text{GL}_n(\z)$ and $\bu \in \z^n$ such that
$$g(\bx) = f(\bx T + \bu).$$
This defines an equivalence relation on the set of quadratic polynomials, and it is clear that the conductor of a quadratic polynomial is a class invariant.
It is also clear that regularity is preserved under this notion of equivalence.   However, simply changing the constant term of a regular quadratic polynomial will produce infinitely many inequivalent regular quadratic polynomials of the same number of variables.    Therefore, in order to obtain any finiteness results analogous to Watson's, we need to confine our attention to a special family of quadratic polynomials.
Following the terminology introduced in \cite{co2}, we call a quadratic polynomial $f(\bx)$ {\em complete} if it takes the form
$$f(\bx) = Q(\bx) + 2B(\bv, \bx) + Q(\bv) = Q(\bx + \bv).$$
By adjusting the constant term and multiplying by a suitable rational number, any nondegenerate quadratic polynomial can be changed to a primitive complete quadratic polynomial.  The main result of this paper is

\begin{thm} \label{mainpolynomial}
Let $c$ be a fixed positive integer.  There are only finitely many equivalence classes of positive primitive ternary regular complete quadratic polynomials with conductor $c$.
\end{thm}

Theorem \ref{mainpolynomial} extends Watson's finiteness result on regular ternary quadratic forms because a quadratic form is equivalent to a complete quadratic polynomial of conductor 1, and each equivalence class of complete quadratic polynomials of conductor 1 contains a unique equivalence class of quadratic forms.

The paper is organized as follows.  Section 2 contains some preliminary results on representations of quadratic polynomials in general.  In Section 3, we discuss various estimates on the number of integers in an interval which satisfy certain arithmetic conditions.  A set of regularity preserving transformations on quadratic polynomials will be introduced in Section 4.  These transformations and their properties will be best described using the language of quadratic spaces, lattices, and cosets.  It is this language that we will adopt for the rest of the paper.  In Sections 5 and 6 we will present the proof of Theorem \ref{mainpolynomial}.  Finally, in Section 7, we will discuss a particular family of quadratic polynomials called polygonal forms, and explain how Theorem \ref{mainpolynomial} implies the finiteness result of regular ternary triangular forms in \cite{co2} mentioned earlier.

\section{Preliminaries}

Let $f(\bx) = Q(\bx) + 2B(\bv, \bx) + m$ be an integral quadratic polynomial. The norm ideal of $Q$, denoted $\norm$, is the ideal of $\z$ generated by the set of integers represented by $Q(\bx)$.  The set $\mathfrak b: = \{2B(\bv, \bx): \bx \in \z^n\}$ is an ideal of $\z$.  It is not hard to check that both $\norm$ and $\mathfrak b$ are inside $\frac{1}{2}\z$ as a result of the integrality of $f(\bx)$, and that $\norm = \frac{1}{2}\z$ if and only if $\mathfrak b = \frac{1}{2}\z$.  We denote the polynomial $Q(\bx) + 2B(\bv, \bx)$ by $f_0(\bx)$, and let $\norm_0(f)$ be the ideal of $\z$ generated by the integers represented by $f_0(\bx)$.

\begin{lem}
If $g(\bx)$ is equivalent to $f(\bx)$, then $\norm_0(f) = \norm_0(g)$.
\end{lem}
\begin{proof}
It is clear that $\norm_0(f) = \norm_0(g)$ if $g(\bx) = f(\bx T)$ for any $T \in \text{GL}_n(\z)$.  Therefore, we may assume that $g(\bx) = f(\bx + \bu)$ for some $\bu \in \z^n$.  A simple calculation shows that $g_0(\bx) = f_0(\bx + \bu) - f_0(\bu)$, whence $\norm_0(g) \subseteq \norm_0(f)$.  The reverse inclusion is obtained by observing that $f(\bx) = g(\bx - \bu)$.
\hfill $\square$\end{proof}

\begin{lem}\label{4c}
If $c$ is the conductor of $f(\bx)$, then $4c \,Q(\bv)\in \norm_0(f)$.  If, in addition, $f(\bx)$ is primitive and complete, then $4c\in \norm_0(f)$.
\end{lem}
\begin{proof}
Since $c\bv$ and $2c\bv$ are in $\z^n$, both $Q(\bv)(c^2 + 2c)$ and $Q(\bv)(4c^2 + 4c)$ are in $\norm_0(f)$.  Thus, $4c\, Q(\bv) \in \norm_0(f)$.  If $f(\bx)$ is primitive and complete, then $Q(\bv)$ is the constant term of $f(\bx)$ and  is relatively prime to $\norm_0(f)$.  This implies the second assertion.
\hfill $\square$\end{proof}

Let $p$ be a prime and $I$ be an ideal of $\z_p$.  We say that $f(\bx)$ represents a coset modulo $I$ if over $\z_p$, $f(\bx)$ represents $r + I$ for some $r \in \z_p$.  Note that for every $\ba \in \z_p^n$, $f_0(\bx)$ represents every $p$-adic integer that is represented by the one variable polynomial $Q(\ba)x^2 + 2B(\bv, \ba)x$ over $\z_p$.

\begin{lem} \label{localrep}
Let $p$ be a prime.  If $f(\bx)$ is primitive,  then $f(\bx)$ represents a coset modulo $p^k\z_p$, where $0 \leq k \leq 1 + \ord_p(\norm_0(f)) + 2\delta_{2,p}$ and $\delta_{2,p}$ is the Kronecker delta.
\end{lem}
\begin{proof}
We may assume that $f(\bx) = f_0(\bx)$ and $\norm_0(f) = \z$.  Suppose first that $p$ is an odd prime.  If $p\mid \mathfrak b$, then there exists $\ba \in \z_p^n$ such that $2B(\bv, \ba) \in p\z_p$ and $Q(\ba) \in \z_p^\times$.  By the Local Square Theorem \cite[63:1]{om} and \cite[63:8]{om}, the polynomial $Q(\ba)x^2 + 2B(\bv, \ba) x$ represents the coset $Q(\ba) + p\z_p$.  If, on the other hand, $2B(\bv, \bb) \in \z_p^\times$ for some $\bb \in \z_p^n$, then by the Local Square Theorem and \cite[63:8]{om} again the polynomial $Q(\bb)x^2 + 2B(\bv, \bb)x$ represents all of $p\z_p$.

The argument when $p = 2$ is along the same line and uses both the Local Square Theorem and \cite[63:8]{om} in the same manner.  If $\mathfrak b = \frac{1}{2}\z$, then there exists $\ba \in \z_2^n$ such that both $2Q(\ba)$ and $4 B(\bv, \ba)$ are in $\z_2^\times$.  In this case, the polynomial $Q(\ba)x^2 + 2B(\bv, \ba)x$ represents every element in $\z_2$.  Suppose that $\mathfrak b = \z$.  Then, $2B(\bv, \bb) \in \z_2^\times$ for some $\bb \in \z_2^n$, and the polynomial $Q(\bb)x^2 + 2B(\bv, \bb)x$ represents $2\z_2$.

We are left with the case when $2 \mid \mathfrak b$.  In this case, there is a vector $\bc \in \z_2^n$ such that $2B(\bv, \bc)$ is divisible by 2 but $Q(\bc)$ is in $\z_2^\times$.  We further divide the discussion into three subcases: $\ord_2(2B(\bv, \bc)) = 1$, $\ord_2(2B(\bv, \bc)) = 2$, or $\ord_2(2B(\bv, \bc)) \geq 3$.  Let us look at the subcase $\ord_2(2B(\bv, \bc)) = 2$.  If $\epsilon \equiv Q(\bc) + B(\bv, \bc)^2Q(\bc)^{-1}$ mod 8, then
\begin{eqnarray*}
B(\bv, \bc)^2 + Q(\bc)\epsilon \equiv Q(\bc)^2 + 2B(\bv, \bc)^2 \equiv Q(\bc)^2 \mbox{ mod } 8
\end{eqnarray*}
Therefore, $B(\bv, \bc)^2 + Q(\bc)\epsilon$ is the square of a 2-adic unit, and hence the polynomial $Q(\bc)x^2 + 2B(\bv, \bc) x$ represents $\epsilon$ over $\z_2$.  This shows that $f(\bx)$ represents a coset modulo $8\z_2$.  The other two subcases are treated similarly and we leave their proofs to the readers.
\hfill $\square$\end{proof}

\begin{rmk}
If $f(\bx)$ is primitive and complete, then by Lemma \ref{4c} the integer $k$ in Lemma \ref{localrep} is bounded above by a constant depending only on $c$.
\end{rmk}

\begin{defn}
A positive quadratic polynomial $f(\bx)$ is called {\em Minkowski reduced}, or simply {\em reduced}, if its quadratic part $Q(\bx)$ is Minkowski reduced\footnote{See \cite[Chapter 12]{ca} for the definition of Minkowski reduced quadratic forms.} and $f(\bx)$ itself attains its minimum at the zero vector.
\end{defn}

Since every positive definite quadratic form is equivalent to a Minkowski reduced quadratic form \cite[Chapter 12, Theorem 1.1]{ca}, it follows easily that every positive quadratic polynomial is equivalent to a reduced quadratic polynomial (see \cite[Lemma 2.2]{co2} for the case of ternary quadratic polynomials).

If $f(\bx) = Q(\bx) + 2B(\bv, \bx) + m$ is a reduced quadratic polynomial and $\{\be_1, \ldots, \be_n\}$ is the standard basis of $\z^n$, then $Q(\bx)$ is Minkowski reduced, $2\vert B(\bv, \be_i)\vert \leq Q(\be_i)$ for  $i = 1,\ldots, n$, and $m$ is the smallest integer represented by $f(\bx)$.  In the special case
when $f(\bx)$ is ternary, $Q(\be_1) \leq Q(\be_2) \leq Q(\be_3)$ are the successive minima of $Q(\bx)$ \cite[Page 285]{vander}.

\begin{prop} \label{inequality}
Suppose that $f(\bx)$ is a positive reduced ternary quadratic polynomial.  Let $\mu_1, \mu_2, \mu_3$ be the successive minima of the quadratic part of $f(\bx)$, and $\ba = (a_1, a_2, a_3)$ be a vector in $\z^3$.
\begin{enumerate}
\item[(a)] If $|a_3| \geq 9$, then $f(\ba)\geq \frac{3}{2} \mu_3$.

\item[(b)] If $|a_3| \leq 8$ and $|a_2| \geq 22$, then $f(\ba) \geq \frac{7}{2} \mu_2$.

\item[(c)] If $|a_3| \leq 8$ and $|a_2| \leq 21$, then $f(\ba) \geq \mu_1(a_1^2 -30|a_1|)$.

\item[(d)] If $|a_3| \leq 8$,  $|a_2| \leq 21$, and $|a_1| \geq 31$, then $f(\ba) \geq 31 \mu_1$.
\end{enumerate}
Consequently,
$$f(\ba) \geq \min \left\{\frac{3}{2} \mu_3,\, \frac{7}{2} \mu_2,\, 31\mu_1\right\}$$
unless $|a_1| \leq 30$,  $|a_2| \leq 21$, and $|a_3| \leq 8$.
\end{prop}
\begin{proof}
Let $f_0(\bx)$ be the quadratic polynomial obtained by taking away the constant term from $f(\bx)$.  Then $f_0(\bx)$ is also reduced and $f(\bx) \geq f_0(\bx)$ for all $\bx \in \z^3$.  Therefore, we may assume that the minimum of $f(\bx)$ is 0.  The proof for this special case can be easily extracted from the proof of \cite[Theorem 1.1]{co2} (see \cite[Page 35]{co2} in particular).
\hfill $\square$\end{proof}

\section{Some Technical Lemmas}

As is in the proof of the finiteness of regular ternary triangular forms \cite[Theorem 1.2]{co2}, we need estimates of the number of integers in an interval that satisfy various local conditions.

\begin{lem}\cite[Lemma 3.4]{co2} \label{kkolemma}
Let $T$ be a finite set of primes and $a$ be an integer not divisible by any prime in $T$.  For any integer $d$, the number of integers in the set $\{d, a + d, \ldots, (n-1)a + d \}$ that are not divisible by any prime in $T$ is at least
$$n\frac{\tilde{p}-1}{\tilde{p} + t - 1} - 2^t + 1,$$
where $t = \vert T \vert$ and $\tilde{p}$ is the smallest prime in $T$.
\end{lem}

Let $\chi_1, \ldots, \chi_r$ be quadratic characters modulo $k_1, \ldots, k_r$, respectively, $\Gamma$ be the least common multiple of $k_1, \ldots, k_r$, and $u_1, \ldots, u_r$ be values taken from the set $\{\pm 1\}$.  Given a nonnegative number $s$ and a positive number $H$, let $\mathcal S_s(H)$ be the number of integers $n$ in the interval $(s, s + H)$ which satisfy the conditions
\begin{equation}\label{condition1}
\chi_i(n) = u_i \quad \mbox{ for } i = 1, \ldots, r \quad \text{ and } \quad \gcd(n, X) = 1,
\end{equation}
where $X$ is a positive integer relatively prime to $\Gamma$.

Let $I$ be a sequence of parameters.  An inequality of the form $A \ll_I B$ will mean that there exists a constant $\kappa$, depending only on the parameters in $I$, such that $\vert A \vert \leq  \kappa B$.  Alternatively, we will write $A = B + O_I(C)$ if $A - B \ll_I C$.  If $I$ is empty, then we will simply use $\ll$ and $O$ instead, and the implied constant $\kappa$ in this case will be an absolute constant.

The following proposition is essentially \cite[Proposition 3.6]{co2}.\footnote{In the statement of \cite[Proposition 3.6]{co2}, it requires $r \leq \omega(\Gamma) + 1$, where $\omega(\Gamma)$ is the number of distinct prime divisors of $\Gamma$.  However, as the referee suggested to us, this inequality holds as a consequence of the independence of the characters.}

\begin{prop} \label{eresult}
Suppose that $\chi_1, \ldots, \chi_r$ are independent.  Let $k$ be a fixed positive integer and $h = \min \{H : \mathcal S_s(H) > k \}$.  Then
\begin{equation} \label{e3.2}
\mathcal S_s(H) = 2^{-r}\frac{\phi(\Gamma X)}{\Gamma X}H + O_\epsilon\!\!\left(H^{\frac{1}{2}}\Gamma^{\frac{3}{16} + \epsilon} X^\epsilon\right),
\end{equation}
and 
\begin{equation} \label{e3.3}
h \ll_{\epsilon, k} \Gamma^{\frac{3}{8} + \epsilon}X^\epsilon,
\end{equation}
where $\phi$ is Euler's phi-function.
\end{prop}

We will need another similar result which estimates $\mathcal S'_s(H)$,  the number of integers in $(s, s + H)$ which satisfy the following stronger conditions
\begin{equation}\label{condition2}
\chi_i(n) = u_i \quad \mbox{ for } i = 1, \ldots, r \quad \text{ and }\quad  n \equiv \tau \mbox{ mod } X,
\end{equation}
where $\tau$ is a fixed integer relatively prime to $X$.

\begin{prop}\label{charactersum}
Suppose that $\chi_1, \ldots, \chi_r$ are independent.  Let $k$ be a fixed positive integer and $h = \min \{H : \mathcal S'_s(H) > k \}$.  Then
\begin{equation} \label{sum1}
\mathcal S'_s(H) = 2^{-r}\frac{\phi(\Gamma)}{\Gamma X}H + O_\epsilon\!\!\left(H^{\frac{1}{2}}(\Gamma X)^{\frac{3}{16} + \epsilon}\right),
\end{equation}
and 
\begin{equation} \label{sum2}
h \ll_{\epsilon, k} \Gamma^{\frac{3}{8} + \epsilon}X^{\frac{19}{8} + \epsilon},
\end{equation}
where $\phi$ is Euler's phi-function.
\end{prop}
\begin{proof}
By the orthogonality of Dirichlet characters, we can express $\mathcal S'_s(H)$ as
\begin{eqnarray*}
\mathcal S'_s(H) & = & \sum_{s<n<s + H} \prod_{j=1}^r \frac{1}{2}(1 + u_j\chi_j(n)) \frac{1}{\phi(X)} \sum_{\psi \!\!\!\!\mod X} \overline{\psi}(\tau)\psi(n) \\
    & = & \frac{1}{2^r\phi(X)} \sum_{\psi\!\!\!\! \mod X} \overline{\psi}(\tau) \sum_{R\subseteq \{1, \ldots, r\}} \left( \prod_{j \in R} u_j \right) \sum_{s<n<s+H} \psi(n) \prod_{j\in R}\chi_j(n).
\end{eqnarray*}

If we let $\psi_0$ be the principal character modulo $X$, then the contribution of $\psi = \psi_0$ and $R = \emptyset$ is obtained from Proposition \ref{eresult} with $r = 0$ to get
$$\frac{1}{2^r\phi(X)} \underset{(n, \Gamma X) = 1}{\sum_{s<n<s+H}} 1 = \frac{\phi(\Gamma)}{2^r \Gamma X} H + \frac{1}{2^r\phi(X)}O_\epsilon\left( H^{\frac{1}{2}}(\Gamma X)^{\epsilon} \right).$$
In all other cases the characters $\psi \prod_{j\in R} \chi_j$ is nontrivial, so that by Burgess' estimate \cite[Theorem 2]{b} they together contribute
$$\frac{2^r\phi(X)- 1}{2^r\phi(X)}\, O_\epsilon\left( H^{\frac{1}{2}} (\Gamma X)^{\frac{3}{16} + \epsilon}\right)$$
to the count.  Adding this to the contribution of $\psi_0$ gives us \eqref{sum1} which is
$$\mathcal S'_s(H) = \frac{\phi(\Gamma)}{2^r\Gamma X}H + O_\epsilon\!\!\left(H^{\frac{1}{2}}(\Gamma X)^{\frac{3}{16} + \epsilon}\right).$$
From this it is easy to see that $S'_s(H) > k$ whenever 
$$\frac{\phi(\Gamma)}{2^r\Gamma X}H - C\Gamma^{\frac{3}{16} + \epsilon}H^{\frac{1}{2}} - k > 0$$
for some positive constant $C$.
 Let $\omega(\Gamma)$ denote the number of distinct prime divisors of $\Gamma$ and $\tau (\Gamma)$ denote the number of positive divisors of $\Gamma$ and observe that $\frac{\Gamma}{\phi(\Gamma)} \ll_\epsilon \Gamma^\epsilon$, $r \ll \omega(\Gamma)$, and $2^r \ll 2^{\omega(\Gamma)} \leq \tau(\Gamma) \ll_\epsilon \Gamma^\epsilon$.  It is now straight forward to deduce \eqref{sum2}. We leave that to the readers.
\hfill $\square$\end{proof}

\begin{rmk}
	We note that although Polya's estimate of character sum suffices in our present argument, we choose to use Burgess' in order to obtain a sharper estimate on the error terms, so that we have a better idea on how much is needed to improve on the other estimates in order to obtain better results.
\end{rmk}

\section{Watson Transformations}

We will describe a family of regularity preserving transformations on complete quadratic polynomials.    The definition and properties of these transformations are best explained in the geometric language of quadratic spaces and lattices, and this is the language we choose to conduct all our subsequent discussions.  The books \cite{k} and \cite{om} are standard references for quadratic spaces and lattices.    Any other unexplained notations and terminologies used later in this paper can be found in either of them.

Let $R$ be a PID, and $(V,Q)$ be a nondegenerate quadratic space over the field of fractions of $R$.  If $L$ is an $R$-lattice on $V$ and $A$ is a symmetric matrix, we shall write ``$L \cong A$" if $A$ is the Gram matrix for $L$ with respect to some basis of $L$.  The discriminant of $L$, denoted $d(L)$, is defined to be the determinant of one of its Gram matrices.  An $n\times n$ diagonal matrix with $a_1, \ldots, a_n$ as its diagonal entries is written as $\langle a_1, \ldots, a_n\rangle$.

An $R$-coset on $V$ is a set $L + \bv$, where $L$ is an $R$-lattice on $V$ and $\bv$ is a vector in $V$.  We define the discriminant of $L + \bv$ to be $d(L)$, the discriminant of $L$.  The conductor of $L + \bv$ is the fractional ideal $\{a \in R : a\bv \in L\}$.  In the case $R = \z$,  the conductor has a positive generator and we will abuse the terminology and call this number the conductor of $L + \bv$.   The $R$-coset $L + \bv$ is {\em integral} if the fractional ideal generated by $Q(L + \bv)$, denoted $\norm(L + \bv)$, is contained in $R$;  and is {\em primitive} if $\norm(L + \bv) = R$.  Two $R$-cosets $L + \bv$ and $M + \bw$ on $V$ and $W$ respectively are said to be {\em isometric}, written $L + \bv \cong M + \bw$, if there exists an isometry $\sigma: V \longrightarrow W$ such that $\sigma(L + \bv) = M + \bw$.  This is the same as requiring $\sigma(L) = M$ and $\sigma(\bv) \in M + \bw$.  It is clear that the conductor of an $R$-coset is a class invariant.

A $\z$-coset is called {\em positive} if the underlying quadratic space is positive definite.  We say that a rational number $a$ is represented by a $\z$-coset $L + \bv$ if there exists $\ba \in L$ such that $Q(\ba + \bv) = a$.  For each prime $p$, the representation of a $p$-adic number by a $\z_p$-coset is defined in the obvious way.  A rational number $a$ is represented by the genus of a $\z$-coset $L + \bv$ if it is represented by $V_\infty$ and by $L_p + \bv$ for every prime $p$.  The $\z$-coset $L + \bv$ is called {\em regular} if it represents all rational numbers that are represented by its genus.  The readers are referred to \cite{co2} for more discussion on representations of numbers by $\z$-cosets in general.

Let $L + \bv$ be a $\z$-coset on $V$.  Fix a basis $\{\be_1, \ldots, \be_n\}$ of $L$.  For any $(x_1, \ldots, x_n) \in \z^n$, we have
\begin{equation*}
Q(x_1\be_1 + \cdots + x_n\be_n + \bv) = \sum_{i=1}^n\sum_{j = 1}^n B(\be_i, \be_j)x_ix_j + \sum_{\ell = 1}^n 2B(\bv, \be_\ell)x_\ell + Q(\bv)
\end{equation*}
which is an $n$-ary complete quadratic polynomial.  Conversely, given an $n$-ary complete quadratic polynomial $f(\bx) = Q(\bx) + 2B(\bv, \bx) + Q(\bv)$, the set $\z^n + \bv$ is a $\z$-coset on the quadratic space $\q^n$ equipped with the quadratic form $Q$.  Thus we have a correspondence between $\z$-cosets and complete quadratic polynomials.   Under this correspondence,  primitive regular complete quadratic polynomials correspond to primitive regular $\z$-cosets, and the conductor of a complete quadratic polynomial will be the same as the conductor of the associated $\z$-coset.  One can readily check that this correspondence leads to a one-to-one correspondence between isometry classes of $\z$-cosets and equivalence classes of complete quadratic polynomials.

\begin{defn}
Let $L$ be a $\z$-lattice.  For any integer $m$, let
$$\Lambda_m(L) = \{\bx \in L : Q(\bx + \bz) \equiv Q(\bz) \mbox{ mod } m \mbox{ for every } \bz \in L\},$$
and for any prime $p$, let
$$\Lambda_m(L_p) = \{\bx \in L_p : Q(\bx + \bz) \equiv Q(\bz) \mbox{ mod } m \mbox{ for every } \bz \in L_p\}.$$
\end{defn}

\begin{lem} \label{watsonproperties} \cite[Lemma 2.2]{ce}
Let $L$ be a $\z$-lattice, $m$ an integer, and $p$ a prime. Then the following hold.
\begin{enumerate}
\item[(a)] $\Lambda_m(L)$ is a sublattice of $L$ and $\Lambda_m(L_p)$ is a sublattice of $L_p$.

\item[(b)] $\Lambda_m(L_p) = \Lambda_m(L)_p$.

\item[(c)]$\Lambda_m(L_p) = L_p$ whenever $p \nmid m$.

\item[(d)] $\norm(\Lambda_m(L)) \subseteq m\z$ and $\norm(\Lambda_m(L_p)) \subseteq m\z_p$.

\item[(e)] If $\norm(L) \subseteq \z$, then $pL \subseteq \Lambda_p(L)$ and $pL_p \subseteq \Lambda_m(L_p)$.

\item[(f)] If $N$ splits $L_p$ and $\norm(N) \subseteq p\z_p$, then $N \subseteq \Lambda_p(L_p)$.

\end{enumerate}
\end{lem}

\begin{lem}\label{watson2}
Let $L$ be a $\z$-lattice and $p$ be an odd prime.  Suppose that $L_p = M \perp N$ where $M$ is unimodular and $\norm(N) \subseteq p\z_p$.  Then
$\Lambda_p(L_p) = pM \perp N$.  If, in addition, $M$ is anisotropic, then $\Lambda_p(L_p) = \{\bx \in L_p : Q(\bx) \in p\z_p\}$.
\end{lem}
\begin{proof}
The first assertion is \cite[Lemma 2.3]{ce}.  For the second assertion, it follows immediately from the definition of $\Lambda_p(L_p)$ that $\Lambda_p(L_p)$ is a subset of $\{\bx \in L_p : Q(\bx) \in p\z_p\}$.  Conversely, suppose that $\bx \in L_p$ and $Q(\bx) \in p\z_p$.  Write $\bx = \bx_0 + \bx_1$, where $\bx_0 \in M$ and $\bx_1 \in N$.  Assume on the contrary that $\bx_0 \not \in pM$.  Then $\bx_0$ is a maximal vector in $M$, and by \cite[83:17]{om} there exists $\bz \in M$ such that $B(\bx, \bz) = 1$.  The binary sublattice of $M$ spanned by $\bx$ and $\bz$ has discriminant in $-\z_p^{\times 2}$, and hence it is isotropic.  This contradicts the hypothesis, thus $\bx_0 \in pM$ and $\bx \in pM \perp N = \Lambda_p(L_p)$.
\hfill $\square$\end{proof}

Suppose that $L$ is a $\z$-lattice on a nondegenerate quadratic space $V$.  Let $p$ be an odd prime such that $p \nmid \norm(L)$.  By Lemma \ref{watson2},
$$p^2\norm(L) \subseteq \norm(\Lambda_p(L)) \subseteq p\norm(L).$$
We denote by $\lambda_p$ the mapping that sends $L$ to the following lattice on the scaled space $V^{\frac{1}{p}}$ or $V^{\frac{1}{p^2}}$:
\begin{equation}\label{norm}
\lambda_p(L) = \begin{cases}
\Lambda_p(L)^{\frac{1}{p}} & \mbox{ if } \norm(\Lambda_p(L)) = p\norm(L),\\
\Lambda_p(L)^{\frac{1}{p^2}} & \mbox{ if } \norm(\Lambda_p(L)) = p^2\norm(L).
\end{cases}
\end{equation}
Collectively, these $\lambda_p$ are what we will refer to as Watson transformations.  Note that $\norm(\lambda_p(L)) = \norm(L)$.

\begin{lem} \label{reducedisc} \cite[Lemma 2.5]{ce}
Suppose that $L$ is a ternary $\z$-lattice on a nondegenerate quadratic space.  If $p$ is an odd prime such that $p^2 \mid d(L)$, then $d(\lambda_p(L)) = \frac{1}{p^t}d(L)$ for some $t \in \{1,2,4\}$.
\end{lem}

\begin{defn}
An integral $\z$-coset $L + \bv$ is said to {\em behave well}\footnote{When $L + \bv$ is a $\z$-lattice, our definition of ``behaves well" is slightly different from the one used in \cite{ce}.} at a prime $p$ if $L_p$ has a unimodular Jordan component of rank at least 2.
\end{defn}

By \cite[92:1b]{om}, $L_p$ represents all $p$-adic units if $L + \bv$ behaves well at $p$.

\begin{prop} \label{descend}
Let $L + \bv$ be a primitive regular ternary $\z$-coset with conductor $c$, and $p$ be an odd prime which does not divide $c$.  If $L + \bv$ does not behave well at $p$, then there exists $\bw$ in the quadratic space underlying $\lambda_p(L)$ such that $\lambda_p(L) + \bw$ is a primitive regular $\z$-coset with conductor $c$.
\end{prop}
\begin{proof}
Let $j$ be the order of $p$ modulo $c$.  We claim that
\begin{equation} \label{lambdashape}
\Lambda_p(L)_q + p^j\bv = \begin{cases}
L_q + \bv & \mbox{ if } q\mid c,\\
\Lambda_p(L)_q & \mbox{ if } q = p,\\
L_q & \mbox{ if } q \nmid pc.
\end{cases}
\end{equation}
The first and the third cases in \eqref{lambdashape} are straightforward.  As for the case $p = q$, since $L + \bv$ does not behave well at $p$, it follows from Lemma \ref{watson2} that $\Lambda_p(L)_p = \{\bx \in L_p : Q(\bx) \in p\z_p\}$.  Therefore, $p^j\bv \in \Lambda_p(L)_p$, and hence $\Lambda_p(L)_p + p^j\bv = \Lambda_p(L)_p$.

Suppose that $a$ is represented by the genus of $\Lambda_p(L) + p^j\bv$.  By \eqref{lambdashape}, $a$ is also represented by the genus of $L + \bv$.  Since $L + \bv$ is regular, $a$ is in fact represented by $L + \bv$, which means that there exists $\bx \in L$ such that $Q(\bx + \bv) = a$.   By \eqref{lambdashape} again, $\bx + \bv$ is contained in $\Lambda_p(L)_q + p^j\bv$ for every $q \neq p$.  At $p$, since $a$ is represented by $\Lambda_p(L)_p + p^j\bv = \Lambda_p(L)_p$, $p$ must divide $a$ by Lemma \ref{watson2}.  Thus, $p \mid Q(\bx + \bv)$ and $\bx + \bv$ must be in $\Lambda_p(L)_p$, by Lemma \ref{watson2} one more time.   Altogether we have shown that $\bx + \bv$ is in $\Lambda_p(L)_q + p^j\bv$ for every prime $q$.  So, $\bx + \bv$ is in $\Lambda_p(L) + p^j\bv$, which proves that $\Lambda_p(L) + p^j\bv$  is regular.

It is clear that $\Lambda_p(L) + p^j\bv$ has conductor $c$.  Since the conductor and the regularity of a $\z$-coset are preserved upon scaling of the underlying quadratic form, $\lambda_p(L) + p^j\bv$ is also regular and has conductor $c$.  It remains to show that $\lambda_p(L) + p^j\bv$ is primitive.  The quadratic form on $\lambda_p(L)$ is $\frac{1}{p^i}Q$, where $i = 1$ or 2, see \eqref{norm}.  By \eqref{lambdashape},
\begin{equation*}
\norm(\Lambda_p(L)_q + p^j\bv) = \begin{cases}
\norm(L_q + \bv) = \z_q & \mbox{ if $q \mid c$},\\
\norm(\Lambda_p(L)_p) = p^i\z_p & \mbox{ if $p = q$},\\
\norm(L_q) = \z_q & \mbox{ if $q \nmid pc$}.
\end{cases}
\end{equation*}
Therefore, $\norm(\lambda_p(L) + p^j\bv) = \z$, which is what we need to show.
\hfill $\square$\end{proof}

\section{Bounding the Discriminant I}

Given a positive $\z$-coset $L + \bv$, we can always choose $\bv$ such that
\begin{equation}\label{bv}
Q(\bv) = \min \{Q(\bx + \bv) : \bx \in L\}.
\end{equation}
If, in addition,  $\{\be_1, \ldots, \be_n\}$ is a Minkowski reduced basis of $L$, then the polynomial $Q(x_1\be_1 + \cdots + x_n\be_n + \bv)$ will be a Minkowski reduced quadratic polynomial.  From now on, unless stated otherwise, we always assume that \eqref{bv} holds when we present a positive $\z$-coset in the form $L + \bv$.

\begin{lem} \label{finite}
There are only finitely many isometry classes of integral $\z$-cosets of a fixed rank and discriminant.
\end{lem}
\begin{proof}
Let $n$ and $k$ be fixed integers, and let $L$ be an integral lattice of rank $n$ and discriminant $k$.  If $\bv \in \q L$ and $L + \bv$ is integral, then  $2B(\bv, \bx) \in \z$ for all $\bx \in L$ and hence $2\bv$ is in $L^\#$, the dual of $L$. Since $L^\#/L$ is a finite group of size $k$,  there are only finitely many possible integral $\z$-cosets of the form $L + \bv$.  The lemma is now clear since it is well-known that there are only finitely many isometry classes of integral lattices of rank $n$ and discriminant $k$.
\hfill $\square$\end{proof}

Let $L + \bv$ be a positive $\z$-coset of rank $n$.  The successive minima $\mu_1 \leq \cdots \leq \mu_n$ of $L$ satisfy the inequality \cite[Prop 2.3]{e}
\begin{equation}\label{disc}
d(L) \leq \mu_1\cdots\mu_n.
\end{equation}

Let $c$ be the conductor of $L + \bv$.  For every prime $p$ dividing $2c$, Lemma \ref{localrep} shows that $L_p + \bv$ represents a coset $r_p + p^{k_p}\z_p$, where $k_p$ is a nonnegative integer bounded above by a constant depending only on $c$.  Set
\begin{equation}\label{ar}
a = a(L + \bv): = \displaystyle{\prod_{p\mid 2c}p^{k_p}}, \quad r = r(L + \bv): = \min \{b \in \mathbb N: b \equiv r_p \mbox{ mod } a \text{ for all } p|2c\}.
\end{equation}
Note that both $a$ and $r$ are bounded above by a constant depending only on $c$.

\begin{prop}\label{behavewell}
Let $L + \bv$ be a primitive regular positive ternary $\z$-coset of conductor $c$.  If $L + \bv$ behaves well at all primes $p$ not dividing $2c$, then $d(L)$ is bounded above by a constant depending only on $c$.
\end{prop}
\begin{proof}
Let $\mathfrak T$ be the set of odd primes $p$ such that $p \nmid c$ and $L_p$ does not represent all $p$-adic integers.  Then $\mathfrak T$ is a finite set.  Let $t$ be the size of $\mathfrak T$, $T$ be the product of all primes in $\mathfrak T$, and $\tilde{p}$ be the smallest prime in $\mathfrak T$.  Since $\tilde{p} > 2$,
$$\omega: = \frac{\tilde{p} + t - 1}{\tilde{p} - 1} \leq t + 1.$$
Let $a$ and $r$ be the integers defined as in \eqref{ar}, and $\mathfrak G$ be the set of all positive integers in the congruence class of $r$ mod $a$ that are relatively prime to $T$.

If $p \mid 2c$, then by Lemma \ref{localrep} $L_p + \bv$ represents all integers in $\mathfrak G$.  If $p \nmid 2c$ and $p \not \in \mathfrak T$, then certainly $L_p + \bv$ represents all integers in $\mathfrak G$.  Suppose that $p \in \mathfrak T$.  Then $L_p + \bv = L_p$ behaves well, which means that $L_p + \bv$ represents every $p$-adic unit.  Since $L + \bv$ is regular, we see that $L + \bv$ represents all positive integers in $\mathfrak G$.  We shall use these integers to obtain an upper bound for the product of the successive minima of $L$, and compare this upper bound with $d(L)$ by \eqref{disc}.  Let $\{\be_1, \be_2, \be_3\}$ be a Minkowski basis of $L$.  Then every $\ba \in \z^3$ is a linear combination $a_1\be_1 + a_2\be_2 + a_3\be_3$.

Let $\eta$ be the smallest positive integer such that
\begin{equation}\label{eta}
\eta > \left( 43\cdot 17\cdot (2\, (15 + \sqrt{225 + a\eta + r}) + 1) + 2^t - 1\right) \omega.
\end{equation}
By means of contradiction, let us suppose that $\frac{2}{3}\mu_2 > r + a\eta$.  If $Q(\ba + \bv) \leq r + a\eta$, then Proposition \ref{inequality} (a), (b) and (c) imply that $\vert a_2 \vert \leq 21$, $\vert a_3\vert \leq 8$, and $Q(\ba + \bv) \geq \mu_1(a_1^2 - 30\vert a_1 \vert)$.  Therefore, if $\vert a_1\vert > 15 + \sqrt{225 + r + a\eta}$, then $Q(\ba + \bv) > r + a\eta$.  This shows that the number of integers smaller than $r + a\eta$ which are represented by $L + \bv$ is at most
$$43\cdot 17 \cdot (2\, (15 + \sqrt{225 + r + a\eta}) + 1).$$
However, by Proposition \ref{kkolemma}, the number of integers in $\{r, r + a, \ldots, r + a(\eta - 1)\}$ which are represented by $L + \bv$ is at least
$$\eta\frac{\tilde{p} - 1}{\tilde{p} + t - 1} - 2^t + 1$$
which can be shown to be larger that  $43 \cdot 17 \cdot (2\, (15 + \sqrt{225 + r + a\eta}) + 1)$ using \eqref{eta}.  This is a contradiction, and hence we must have
$$\mu_2 \leq \frac{2}{3}(r + a\eta).$$
A straightforward calculation shows that $\eta \ll_a t 2^t$; thus
\begin{equation}\label{mu2}
\mu_2 \ll_a t 2^t.
\end{equation}

Let $M$ be the sublattice spanned by $\be_1$ and $\be_2$.  Then $d(M) \leq \mu_1\mu_2$.  For each integer $1 \leq j \leq 17$, $Q(y_1\be_1 + y_2\be_2 + (j - 9)\be_3 + \bv)$ is a positive binary integral quadratic polynomial in variables $\by = (y_1, y_2)$ which takes the form
$$h_j(\by) = q(\by) + 2b(\bw_j, \by) + m_j$$
where $q$ is the quadratic form on $M$ and $b$ is the bilinear form associated to $q$.  We will find a positive integer in $\mathfrak G$ which is not represented by any one of these 17 binary quadratic polynomials.  By Proposition \ref{inequality} (a), this integer will lead to an upper bound on $\mu_3$.

Let $D$ be the product of those primes in $\mathfrak T$ that do not divide $d(M)$.  For the sake of convenience, we set $\ell_0$ to be 1. By Proposition \ref{eresult}, for $1\leq i \leq 17$, there exists an integer $n_i$ such that $\left(\frac{-d(M)}{n_i}\right) = -1$ and $n_i$ is relatively prime to $2aT\ell_1...\ell_{i-1}$ (when $i = 1$,  this condition becomes $n_1 \nmid 2aT$).    The multiplicative property of the Jacobi symbol then guarantees the existence of a prime divisor $\ell_i$ of $n_i$ such that $\left(\frac{-d(M)}{\ell_i}\right) = -1$, $\ell_i\nmid 2aT\ell_1\ldots \ell_{i-1}$, and
\begin{eqnarray*}
\ell_i & \ll_\epsilon & d(M)^{\frac{3}{8} + \frac{\epsilon}{17}}(D 2a \ell_1\cdots \ell_{i-1})^{\frac{\epsilon}{17}}\\
    & \ll_{\epsilon, a} & d(M)^{\frac{3}{8} + \frac{\epsilon}{17}}D^{\frac{\epsilon}{17}} (\ell_1\cdots \ell_{i-1})^{\frac{\epsilon}{17}}.
\end{eqnarray*}
Let $\kappa$ be the product $\ell_1\cdots \ell_{17}$.  Then,
\begin{equation*}
\kappa \ll_{\epsilon, a}  d(M)^{\frac{51}{8} + \epsilon} D^\epsilon \prod_{i=1}^{17} (\ell_1\cdots \ell_{i-1})^{\frac{\epsilon}{17}} \leq  d(M)^{\frac{51}{8} + \epsilon} D^\epsilon \kappa^\epsilon
\end{equation*}
and hence
$$\kappa \ll_{\epsilon, a} d(M)^{\frac{51}{8(1 - \epsilon)} + \frac{\epsilon}{1 - \epsilon}} D^{\frac{\epsilon}{1 - \epsilon}}.$$
We now choose $\epsilon$ to be $\frac{1}{4}$.  This leads to
$$\kappa \ll_a d(M)^{\frac{51}{6} + \frac{1}{3}}D^\frac{1}{3} \ll_a (t2^t)^\frac{53}{6}T^\frac{1}{3}.$$

For each $1 \leq j \leq 17$, $M_{\ell_j}$ is a binary unimodular $\z_{\ell_j}$-lattice.  Since $b(\bw_j, \by) \in \z_{\ell_j}$ for all $\by \in M_{\ell_j}$,  $\bw_j$ is in $M_{\ell_j}$ and hence $q(\bw_j) \in \z_{\ell_j}$.  By the Chinese Remainder Theorem, there exists $m \leq \kappa^2$ such that
$$am \equiv \ell_j + m_j - q(\bw_j) - r \mbox{ mod } \ell_j^2 \quad \mbox{ for } 1\leq j \leq 17.$$
Then, for every integer $\lambda$ and every $1 \leq j \leq 17$, $\ord_{\ell_j}((a(m + \lambda \kappa^2)+ r) + q(\bw_j) - m_j) = 1$, and hence $(a(m + \lambda \kappa^2) + r) + q(\bw_j) - m_j$ is not represented by $q(\by + \bw_j)$ over $\z_{\ell_j}$.  In other words, $a(m + \lambda \kappa^2) + r$ is not represented by $h_j(\by)$.  On the other hand, by Lemma \ref{kkolemma} there must be a positive integer $f \leq (t + 1)2^t$ such that $a\kappa^2 f + am + r$ is relatively prime to $T$.  Thus, the integer $a(m + \kappa^2 f) + r$ is represented by $Q(\bx + \bv)$ but not by $h_j(\by)$ for any $j$.  It follows from Proposition \ref{inequality}(a) that there exists $\ba \in L$ such that
$$a(m + \kappa^2 f) + r = Q(\ba + \bv) \geq \frac{3}{2}\mu_3,$$
whence
\begin{equation} \label{mu3}
\mu_3 \leq \frac{2}{3}(a(m + \kappa^2 f) + r) \ll_a (t2^t)^{\frac{53}{3}} T^{\frac{2}{3}}.
\end{equation}

Finally, using the inequality \eqref{disc} and combining \eqref{mu2} and \eqref{mu3}, we have
$$T \leq d(L) \leq \mu_1\mu_2\mu_3 \leq \mu_2^2\mu_3 \ll_a (t2^t)^{\frac{59}{3}} T^{\frac{2}{3}}.$$
Since $T$ grows at least as fast as $t!$, the above inequalities show that $t$, and hence $T$ as well, must be bounded above by a constant depending only on $a$. This means that $d(L)$ is also bounded above by a constant depending only on $a$.  This proves the proposition,  since $a$ is bounded above by a constant depending only on $c$.
\hfill $\square$\end{proof}

\section{Bounding the Discriminant II}

Let $L + \bv$ be a primitive regular positive ternary $\z$-coset of conductor $c$.  We can apply Proposition \ref{descend} repeatedly at suitable odd primes and eventually obtain a primitive regular positive ternary lattice $K + \bw$ which behaves well at all odd primes not dividing $c$.  Moreover, the conductor of $K + \bw$ is still $c$ and $d(K)$ is a divisor of $d(L)$.  Let $\mathfrak L$ be the set of prime divisors of $d(L)$ which do not divide $d(K)$.  By Proposition \ref{behavewell}, every prime divisor of $d(L)$ that does not belong to $\mathfrak L$ is bounded by a constant depending only on $c$.

\begin{prop}\label{bound2}
All the primes in $\mathfrak L$ are bounded above by a constant depending only on $c$.
\end{prop}
\begin{proof}
Let $\ell$ be a prime in $\mathfrak L$ which does not divide $2c$.   Without loss of generality, we may assume that $L + \bv$ does not behave well at $\ell$ but does so at all odd primes not dividing $\ell c$.  Moreover, since successive applications of Proposition \ref{descend} at $\ell$ to $L + \bv$ results in a $\z$-coset which behaves well at $\ell$, we can further assume that $L_\ell$ is isometric to $\langle \alpha, \beta\ell^2, \gamma\ell^2\rangle$, where $\alpha, \beta, \gamma \in \z_\ell^\times$.  Let $a$ and $r$ be the positive integers as defined in \eqref{ar}, and $I$ be the product of prime divisors of $d(L)$ which do not divide $2\ell c$.  It is easy to see that $L + \bv$ represents all positive integers congruent to $r$ mod $a$ that are relatively prime to $I$.  Let $b$ be the gcd of $a$ and $r$.

By applying Proposition \ref{charactersum} to the quadratic residue character mod $\ell$, taking $\epsilon = \frac{1}{32}$ and $X$ to be the product $\frac{a}{b}I$, we see that the number of positive integers less than $m$ which are represented by $L + \bv$ is
\begin{equation}
\frac{b\phi(\ell)}{2a\ell I} m + O_c(m^{\frac{1}{2}}\ell^{\frac{7}{32}}).
\end{equation}
Therefore, there exists a positive constant $N_1$, depending only on $c$, such that the number of positive integers less than $m$ which are represented by $L + \bv$ is at least
\begin{equation} \label{atleast}
\frac{b\phi(\ell)}{2a\ell I} m -  N_1m^{\frac{1}{2}}\ell^{\frac{7}{32}}.
\end{equation}
Now, suppose that $m < \mu_2$, and let $\ba \in L$ such that $Q(\ba + \bv) \leq m$.  We write $\ba$ as $a_1\be_1 + a_2\be_2 + a_3\be_3$, where $\{\be_1, \be_2, \be_3\}$ is a Minkowski reduced basis of $L$.  Then, by Proposition \ref{inequality}, we must have
$$\vert a_3 \vert \leq 8, \quad \vert a_2 \vert \leq 21, \quad \text{and} \quad Q(\ba + \bv) \geq \mu_1(a_1^2 - 30 \vert a_1 \vert).$$
Thus, if $\vert a_1 \vert > 15 + \sqrt{225 + m}$, then $Q(\ba + \bv) > m$  which would be impossible.  Therefore, the number of positive integers less than $m$ which are represented by $L + \bv$ is at most
\begin{equation}\label{atmost}
43\cdot 17 \cdot (2(15 + \sqrt{225 + m}) + 1).
\end{equation}
Combining \eqref{atleast} and \eqref{atmost} together, we obtain the inequality
$$\frac{b\phi(\ell)}{2a\ell I} m -  N_1m^{\frac{1}{2}}\ell^{\frac{7}{32}} \leq 43\cdot 17 \cdot (2(15 + \sqrt{225 + m}) + 1).$$
Using the inequality $\frac{\ell}{\phi(\ell)} \ll \ell^{\frac{1}{32}}$, we deduce that whenever $m < \mu_2$,
$$m \ll_c \ell^{\frac{1}{2}}.$$
This implies that $\mu_2$ itself also satisfies $\mu_2 \ll_c \ell^{\frac{1}{2}}$.

Since $L_\ell \cong \langle \alpha, \beta \ell^2, \gamma \ell^2 \rangle$, we must have
$$\ell^2 \leq \mu_1\mu_2 \leq \mu_2^2 \ll_c \ell$$
and hence $\ell$ is bounded above by a constant depending only on $c$.
\hfill $\square$\end{proof}

We are ready to prove Theorem \ref{mainpolynomial}, which is now restated in the language of $\z$-cosets.

\begin{thm} \label{maincoset}
Let $c$ be a fixed positive integer.  There are only finitely many isometry classes of positive primitive ternary regular $\z$-cosets with conductor $c$.
\end{thm}
\begin{proof}
Let $L + \bv$ be a positive primitive ternary regular $\z$-coset with conductor $c$.  In what follows, when we say that a numerical quantity is bounded, it will be understood that the said quantity is bounded above by a constant depending only on $c$.  By virtue of Lemma \ref{finite} and inequality \eqref{disc}, the theorem will be proved once we show that the successive minima $\mu_1, \mu_2, \mu_3$ of $L$ are all bounded.

Let $t$ be the number of prime divisors of $d(L)$ which do not divide $2c$, and let $T$ be the product of all these primes.  Propositions \ref{behavewell} and \ref{bound2} show that $T$ and $t$ are bounded.  Let $a$ and $r$ be the integers as defined by \eqref{ar}.

As is in the proof of Proposition \ref{bound2}, it follows from Proposition \ref{charactersum} that the number of integers less than $m$ represented by $L + \bv$ is
\begin{equation} \label{main1}
2^{-t} \frac{\phi(T)}{Ta}m + O_c(m^{\frac{1}{2}}).
\end{equation}
At the same time, if $m < \mu_2$, the same number must be at most
\begin{equation}\label{main2}
43\cdot 17\cdot (2(15 + \sqrt{225 + m}) + 1).
\end{equation}
Combine \eqref{main1} and \eqref{main2} together and proceed as in the proof of Proposition \ref{bound2}, we deduce that $\mu_2$ is bounded.  This, of course, implies that $\mu_1$ is also bounded.

To bound $\mu_3$, we proceed as in the proof of Proposition \ref{behavewell} but keep in mind that $t, T, \mu_1, \mu_2, a$, and $r$ are all bounded.  Let $\{\be_1, \be_2, \be_3\}$ be a Minkowski basis of $L$.  Let $M$ be the sublattice spanned by $\be_1$ and $\be_2$.  As in the proof of Proposition \ref{behavewell}, there are 17 primes $\ell_1, \ldots, \ell_{17}$ such that $M_{\ell_j}$ is anisotropic for $1 \leq j \leq 17$.  Let $\kappa$ be the product of these 17 primes, which are bounded.  There exists a positive integer $m \leq \kappa^2$ such that for any integer $\lambda$, $a(m + \lambda\kappa^2) + r$ is not represented by any of the 17 binary quadratic polynomials $Q(y_1\be_1 + y_2\be_2 + (j-9)\be_3 + \bv)$.

Let $d$ be the gcd of $a\kappa^2$ and $am + r$, which is relatively prime to $T$.  By Proposition \ref{charactersum}, there exists a bounded positive integer $n$ such that $dn$ is represented by $L_p$ for all $p \mid T$ and $dn \equiv am + r$ mod $a\kappa^2$.  Then \eqref{sum2} guarantees that $dn$ is a bounded integer represented by $L + \bv$, and is of the form $a(m + \lambda\kappa^2) + r$ for some integer $\lambda$.  It follows from Proposition \ref{inequality}(a) that $\mu_3 \leq \frac{2}{3}dn$, which means that $\mu_3$ is bounded.
\hfill $\square$\end{proof}

\section{Polygonal Forms}

In this section, we will explain how either Theorem \ref{mainpolynomial} or Theorem \ref{maincoset} implies the finiteness result of regular ternary triangular forms proved in \cite{co2}.  But instead of focusing on triangular forms, we will broaden our discussion to include a wider class of quadratic polynomials.

Let $m \geq 3$ be a positive integer.  The set of (generalized) $m$-gonal numbers is the set of integers represented by the integral quadratic polynomial
$$\frac{(m-2)x^2 - (m - 4)x}{2}.$$
For examples, when $m = 4$ these numbers are precisely the squares of integers, and the case $m = 3$ gives us the triangular numbers.   Given any positive integers $a_1, \ldots, a_n$, we call the polynomial
\begin{equation}\label{mgonal}
\Delta_m(a_1, \ldots, a_n): = \sum_{i=1}^n a_i \left(\frac{(m-2)x_i^2 - (m - 4)x_i}{2} \right)
\end{equation}
an $m$-gonal form (in $n$ variables).  When $m = 3$, this is exactly a triangular form.  After completing the squares, \eqref{mgonal} becomes
$$\Delta_m(a_1, \ldots, a_n) = \sum_{i= 1}^n a_i \frac{m-2}{2} \left (x_i - \frac{m-4}{2(m-2)}\right)^2 - \frac{(m-4)^2}{8(m-2)}\sum_{i=1}^n a_i.$$
It follows that the conductor of $\Delta_m(a_1, \ldots, a_n)$ is
$$\begin{cases}
2(m-2) & \mbox{ if $m$ is odd},\\
m-2 & \mbox{ if $\ord_2(m) = 1$},\\
\frac{m-2}{2} & \mbox{ if $\ord_2(m) > 1$}.
\end{cases}$$
Therefore, for a fixed $m$, the conductor of $\Delta_m(a_1, \ldots, a_n)$ is the same for any choices of $a_1, \ldots, a_n$.

Let $L$ be the $\z$-lattice which is isometric to $\langle 4(m-2)^2a_1, \ldots, 4(m-2)^2a_n\rangle$ with respect to some basis $\{\be_1, \ldots, \be_n\}$, and $\bv$ be the vector $-\frac{m-4}{2(m-2)}(\be_1 + \cdots + \be_n)$.  Then an integer $k$ is represented by $\Delta_m(a_1, \ldots, a_n)$ if and only if $L + \bv$ represents $8(m-2)k + (m-4)^2(a_1 + \cdots + a_n)$.  In particular, $\Delta_m(a_1, \ldots, a_n)$ is regular if and only if $L + \bv$ is regular.  Equivalent $m$-gonal forms will lead to isometric $\z$-cosets under this correspondence.  However, we have the following simple lemma about equivalent $m$-gonal forms.

\begin{lem}\label{simple}
If $\Delta_m$ and $\Delta_m'$ are equivalent $m$-gonal forms, then up to a permutation of the variables $\Delta_m = \Delta_m'$.
\end{lem}
\begin{proof}
Suppose that $\Delta_m$ and $\Delta_m'$ are equivalent $m$-gonal forms in $n$ variables.  Let $Q$ and $Q'$ be the quadratic part of $\Delta_m$ and $\Delta_m'$ respectively.   Then $Q$ and $Q'$ are equivalent quadratic forms, and hence $Q$ and $Q'$ have the same successive minima.  However, both $Q$ and $Q'$ are diagonal quadratic forms, and a diagonal quadratic form is determined, up to a permutation of the variables, by its successive minima.  Therefore, after a suitable permutation of the variables, we may assume that $Q = Q'$.  Since an $m$-gonal form is completely determined by the coefficients of its quadratic part, we may conclude that $\Delta_m$ and $\Delta_m'$ must be equal.
\hfill $\square$\end{proof}

As a corollary of Lemma \ref{simple} and Theorem \ref{maincoset}, we obtain the following finiteness result for regular ternary $m$-gonal forms which includes the case of triangular forms as a special case.

\begin{cor}
For every $m \geq 3$, there are only finitely many primitive regular ternary $m$-gonal forms.
\end{cor}

\section*{Acknowledgements}

The authors would like to thank the referee for his/her insightful comments and suggestions.

\section*{References}
\bibliographystyle{elsart-num-sort}
\bibliography{Chan_Ricci_Bib.bib}
\end{document}